\documentclass[a4paper, 12pt]{article}
\usepackage{latexsym,amsfonts,palatino,eulervm,amsthm,graphicx,verbatim, float, mathdots, bm}
\usepackage[margin=2cm]{geometry}
\usepackage{amsmath}
\usepackage{algorithm}
\usepackage[noend]{algpseudocode}
\usepackage{amsmath,amsthm}
\usepackage{amsfonts}
\usepackage{multicol}
\usepackage{multirow,bigdelim}
\usepackage{fancyvrb}
\usepackage[affil-it]{authblk}
\usepackage[hidelinks]{hyperref}
\usepackage{xcolor}

\usepackage{mathtools}

\newcommand\SmallArray[1]{{%
  \tiny \arraycolsep=0.08\arraycolsep\ensuremath{\begin{array}{c}#1\end{array}}}}

\newcommand\Matrix[1]{{%
  \small \arraycolsep=0.8\arraycolsep\ensuremath{\begin{pmatrix}#1\end{pmatrix}}}}

\makeatletter
\def\BState{\State\hskip-\ALG@thistlm}
\makeatother

\algnewcommand\And{\textbf{and} }

\newtheorem{theorem}{Theorem}[section]
\newtheorem{lemma}[theorem]{Lemma}
\newtheorem{corollary}[theorem]{Corollary}
\newtheorem{conjecture}[theorem]{Conjecture}

\theoremstyle{definition}
\newtheorem{definition}[theorem]{Definition}
\newtheorem{example}[theorem]{Example}

\title{Weighted projections of alternating sign matrices: Latin-like squares and the ASM polytope}
\date{}
\author{Cian O'Brien \\ \href{mailto:obrien.cian@outlook.com}{obrien.cian@outlook.com}}
\affil{Department of Mathematics \& Computer Studies\\ Mary Immaculate College\\Limerick, Ireland}

\begin{document}
\setlength{\parindent}{0pt}
\setlength{\parskip}{1ex}

\maketitle

\begin{abstract}
The weighted projection of an alternating sign matrix (ASM) was introduced by Brualdi and Dahl (2018) as a step towards characterising a generalisation of Latin squares they introduced using alternating sign hypermatrices. If $z_n = (n,\dots,2,1)$, then the weighted projection of an ASM $A$ is equal to $z_n^TA$. Brualdi and Dahl proved that the weighted projection of an $n \times n$ ASM is majorized by the vector $z_n$, and conjectured that any positive integer vector majorized by $z_n$ is the weighted projection of some ASM. The main result of this paper presents a proof of this conjecture, via monotone triangles. A relaxation of a monotone triangle, called a row-increasing triangle, is introduced. It is shown that for any row-increasing triangle $T$, there exists a monotone triangle $M$ such that each entry of $M$ occurs the same number of times as in $T$. A construction is also outlined for an ASM with given weighted projection. The relationship of the main result to existing results concerning the ASM polytope $ASM_n$ is examined, and a characterisation is given for the relationship between elements of $ASM_n$ corresponding to the same point in the permutohedron of order $n$. Finally, the limitations of the main result for characterising alternating sign hypermatrix Latin-like squares is considered.

\phantom{a}


\end{abstract}

\section{Introduction}

An \emph{alternating sign matrix (ASM)} is a ${(0,1,-1)}$-matrix for which, in each row and column, the non-zero elements alternate in sign and sum to 1. An immediate consequence of this definition is that ASMs are square matrices with the first and last non-zero entry of each row and column equal to $+1$. Permutation matrices are examples of ASMs, and ASMs arise naturally as the unique minimal lattice extension of the permutation matrices under the \emph{Bruhat order} \cite{lattice}. 

A \emph{Latin square} of order $n$ is an $n \times n$ array containing $n$ symbols such that each symbol occurs exactly once in each row and column. Any $n \times n$ Latin square $L$ with symbols $1, 2, \dots, n$ can be decomposed into a unique sequence $P$ of $n \times n$ mutually orthogonal permutation matrices $P = P_1, P_2, \dots, P_n$ by the following relation.

\[L = L(P) = \sum_{k=1}^n kP_k\]

For example, consider the following Latin square and its decomposition.

\[\Matrix{
    1 & 2 & 3 \\
    2 & 3 & 1 \\
    3 & 1 & 2} 
=
1\Matrix{
      1 & 0 & 0 \\
      0 & 0 & 1 \\
      0 & 1 & 0}
+2\Matrix{
      0 & 1 & 0 \\
      1 & 0 & 0 \\
      0 & 0 & 1}
+3\Matrix{
      0 & 0 & 1 \\
      0 & 1 & 0 \\
      1 & 0 & 0}\]

This interpretation leads very natrually to a generalisation of Latin squares, first introduced by Brualdi and Dahl \cite{brualdidahl}, obtained by replacing the sequence of permutation matrices with planes of an \emph{alternating sign hypermatrix (ASHM)}. Before we define these objects, we first define some features of a hypermatrix.

An $n \times n \times n$ hypermatrix $A = [a_{ijk}]$ has has $n^2$ lines of each of the $3$ following types. Each line has $n$ entries.
\begin{itemize}
\item \emph{Row lines} $A_{*jk} = [a_{ijk} : i = 1, \dots, n]$, for given $1 \leq j,k \leq n$;
\item \emph{Column lines} $A_{i*k} = [a_{ijk} : j = 1, \dots, n]$, for given $1 \leq i,k \leq n$;
\item \emph{Vertical lines} $A_{ij*} = [a_{ijk} : k = 1, \dots, n]$, for given $1 \leq i,j \leq n$.
\end{itemize}

We refer to a \emph{plane} $P_k(A)$ of $A$ to be the horizontal plane $A_{**k} = [a_{ijk}: 1 \leq i,j \leq n ]$, for given $k \in \{1,\dots,n\}$.

An \emph{alternating sign hypermatrix (ASHM)} is a $(0,1,-1)$-hypermatrix for which the non-zero entries in each row, column, and vertical line of the hypermatrix alternate in sign, starting and ending with $+1$.

For example, the following is a $3 \times 3 \times 3$ ASHM.

\[A = \Matrix{
      0 & 0 & 1 \\
      0 & 1 & 0 \\
      1 & 0 & 0}
\hspace{-0.1cm}\nearrow\hspace{-0.1cm}\Matrix{
      0 & 1 & 0 \\
      1 & -1 & 1 \\
      0 & 1 & 0}
\hspace{-0.1cm}\nearrow\hspace{-0.1cm}\Matrix{
      1 & 0 & 0 \\
      0 & 1 & 0 \\
      0 & 0 & 1}\]

The north-east arrow $P_k(A) \nearrow P_{k+1}(A)$ is used to denote that $P_k(A)$ is below $P_{k+1}(A)$. For simplicity, for the remainder of this paper, we omit the zero entries of all ASHMs, and represent all non-zero entries with $+$ or $-$.

An \emph{alternating sign hypermatrix Latin-like square (ASHL)} is an $n \times n$ array $L$ constructed from an $n \times n \times n$ ASHM $A$ by 
\[L = L(A) = \sum_{k=1}^n kP_k(A)\text{.}\]

From the previous example, we then have the following ASHL.
\[L(A) = 1\Matrix{
       &  & + \\
       & + &  \\
      + &  & }
+2\Matrix{
       & + &  \\
      + & - & + \\
       & + & }
+3\Matrix{
      + &  &  \\
       & + &  \\
       &  & +}
=\Matrix{
    3 & 2 & 1 \\
    2 & 2 & 2 \\
    1 & 2 & 3}\]

For compactness, we introduce an alternative notation, representing an $n \times n \times n$ ASHM by an $n \times n$ grid in which each entry of the grid contains the weighted sum for the corresponding vertical line of the ASHL (without adding together the terms of the sum). Using this notation, we represent the ASHM from above in the following way.

\[A = \begin{array}{|c|c|c|}
\hline
    3 & 2 & 1 \\
\hline
    2 & \SmallArray{1-2\\+3} & 2 \\
\hline
    1 & 2 & 3 \\
\hline
\end{array}
\hspace{2cm}
L(A) = \begin{array}{|c|c|c|}
\hline
    3 & 2 & 1 \\
\hline
    2 & 2 & 2 \\
\hline
    1 & 2 & 3 \\
\hline
\end{array}\]

It can be easily read that this is an ASHM using this notation, in the following way.
\begin{itemize}
\item The signs of the terms in each entry must alternate in sign, beginning and ending with $+$, when ordered in increasing order.
\[\text{For example, the middle entry from $A$ above is }+1-2+3\]
\item For each integer $1,\dots,n$, the positions of the signs corresponding to that integer must form an ASM.
\[\text{From the previous example, we have the following for integers 1, 2, and 3.}\]
\[\begin{array}{|c|c|c|}
\hline
     &  & +1 \\
\hline
     & +1 &  \\
\hline
    +1 &  &  \\
\hline
\end{array}
\hspace{1cm}
\begin{array}{|c|c|c|}
\hline
     & +2 &  \\
\hline
    +2 & -2 & +2 \\
\hline
     & +2 &  \\
\hline
\end{array}
\hspace{1cm}
\begin{array}{|c|c|c|}
\hline
    +3 &  &  \\
\hline
     & +3 &  \\
\hline
     &  & +3 \\
\hline
\end{array}\]
\end{itemize}

As well as introducing ASHMs and ASHLs, Brualdi and Dahl \cite{brualdidahl} posed a number of interesting problems about these objects. Some of these questions are addressed in \cite{ASHM-decomp}, firstly by characterising how pairs of ASHMs with the same corresponding ASHL relate to one another, and secondly by exploring the maximum number of times a particular integer may occur as an entry of an $n \times n$ ASHL.

This paper addresses a further question posed in \cite{brualdidahl}, which are motivated by the problem of characterising ASHLs without needing to find an underlying ASHM. Brualdi and Dahl \cite{brualdidahl} introduce the \emph{weighted projection} of an ASM as an intermediate step towards characterising ASHLs. They prove that the weighted (vertical or horizontal) projection of an $n \times n$ ASM is majorized by the vector $(n,\dots,2,1)$, and they conjecture that any vector majorized by this is the weighted projection of some $n \times n$ ASM. In the next section, we prove this conjecture using monotone traingles, which are known to be in bijection with ASMs, and outline an explicit construction for an ASM with given weighted projection.

The third section examines the relationship between our main result and existing results concerning the ASM polytope $ASM_n$. A characterisation of elements of $ASM_n$ with the same weighted projection is given. The final section of this paper returns to the question of characterising ASHLs, discussing the limitations of the main result and describing what is known to date.


\section{Weighted projections of alternating sign matrices}
Brualdi and Dahl \cite{brualdidahl} introduce the weighted vertical and horizontal projections of an ASM. Since any action of the dihedral group $D_4$ on an ASM results in another ASM, these concepts are equivalent. We therefore only consider the weighted vertical projection of an ASM, hereby referred to as \emph{the} weighted projection.

\begin{definition}Let $A$ be an $n \times n$ ASM. The \emph{weighted projection} of $A$ is the vector $v(A) = z_n^TA$, for $z_n = (n, n-1, \dots, 3, 2, 1)$.\end{definition}

For example, the following ASM
\[A = \begin{pmatrix}
&+&&\\
+&-&+&\\
&+&&\\
&&&+
\end{pmatrix}\]
has weighted projection $v(A) = (4,3,2,1)^TA = (3,3,3,1)$.

The weighted projection of an ASM is a 2-dimensional analogue of the operation $L(A)$, which results in an ASHL from an ASHM $A$. Therefore any row or column of an ASHL is the weighted projection of some ASM. We can therefore use a characterisation of weighted projections of ASMs as a step towards characterising ASHLs.

\begin{definition}A monotone triangle of order $n$ is a triangle with $n$ rows and entries from $\{1,2,\dots,n\}$ such that the entries in each row are strictly increasing, and the entries in each north-east and south-east line are weakly increasing.\end{definition}

There exists a bijection between monotone triangles and ASMs of order $n$ \cite{monotone bib}. To obtain a monotone triangle from an ASM $A$, first produce the matrix for which every entry is the sum of entries in the corresponding column of $A$, above and including the entry itself. The position of $+$ entries in each row of the resulting matrix indicate the entries of the monotone triangle.

\[A = \begin{pmatrix}
&+&&\\
+&-&+&\\
&+&&\\
&&&+
\end{pmatrix}
\hspace{0.3cm}\rightarrow\hspace{0.3cm}
\begin{pmatrix}
&+&&\\
+&&+&\\
+&+&+&\\
+&+&+&+
\end{pmatrix}
\hspace{0.3cm}\rightarrow\hspace{0.3cm}
\begin{array}{ccccccc}
&\hspace{-0.2cm}&\hspace{-0.2cm}&\hspace{-0.2cm}2&\hspace{-0.2cm}&\hspace{-0.2cm}&\hspace{-0.2cm}\\
&\hspace{-0.2cm}&\hspace{-0.2cm}1&\hspace{-0.2cm}&\hspace{-0.2cm}3&\hspace{-0.2cm}&\hspace{-0.2cm}\\
&\hspace{-0.2cm}1&\hspace{-0.2cm}&\hspace{-0.2cm}2&\hspace{-0.2cm}&\hspace{-0.2cm}3&\hspace{-0.2cm}\\
1&\hspace{-0.2cm}&\hspace{-0.2cm}2&\hspace{-0.2cm}&\hspace{-0.2cm}3&\hspace{-0.2cm}&\hspace{-0.2cm}4
\end{array}\]

We call this intermediate matrix the \emph{partial sum matrix} $p(A)$, and note that $p(A)$ is always a $(0,1)$-matrix for ASM $A$, since any $-1$ entry of $A$ has a $+1$ above it as the previous non-zero entry.

It follows readily from these definitions that the $j$th entry of $v(A)$ is equal to the number of $+$ entries in column $j$ of $p(A)$, and hence the number of occurrences of $j$ as an entry in the monotone triangle corresponding to $A$. Recall, in the example above, that $v(A) = (3,3,3,1)$ and note that 1, 2, and 3 occur three times each in the corresponding monotone triangle, while 4 occurs once.

These results about monotone triangles and partial sum matrices provide the necessary tools to prove Brualdi and Dahl's conjecture about weighted projections of ASMs and \emph{majorization}.

\begin{definition} A vector $x = (x_1, x_2, \dots, x_n)$ is \emph{majorized} by a vector $y = (y_1, y_2, \dots, y_n)$, denoted $x \preceq y$, if for all $k = 1, 2, \dots, n$,
\[\sum_i^k x^{\downarrow}_i \leq \sum_i^k y^{\downarrow}_i,\]
where $x^{\downarrow}_i$ is the $i^{\text{th}}$ highest valued entry of $x$ (and similar for $y$), with equality for $k = n$.\end{definition}

The following is a conjecture of Brualdi and Dahl \cite{brualdidahl}.
\begin{conjecture}\label{proj_conj}
Let $v$ be a vector of positive integers such that $v \preceq (n,n-1,\dots,2,1)$. Then there exists an $n\times n$ ASM whose weighted projection is $v$.
\end{conjecture}

We prove this conjecture using the following two lemmas and a relaxation of a monotone triangle called a \emph{row-increasing triangle}, which does not have the requirement that entries are weakly increasing along north-east and south-east lines. We note that any $n \times n$ $(0,1)$-matrix with row-sum vector $(1,2,\dots,n)$ corresponds to a unique row-increasing triangle, since the row-sum condition guarantees exactly $k$ entries in row $k$ of the triangle.

For example, 

\[\begin{pmatrix}
+&&&\\
&+&+&\\
+&+&+&\\
+&+&+&+
\end{pmatrix}
\hspace{0.3cm}\rightarrow\hspace{0.3cm}
\begin{array}{ccccccc}
&\hspace{-0.2cm}&\hspace{-0.2cm}&\hspace{-0.2cm}1&\hspace{-0.2cm}&\hspace{-0.2cm}&\hspace{-0.2cm}\\
&\hspace{-0.2cm}&\hspace{-0.2cm}2&\hspace{-0.2cm}&\hspace{-0.2cm}3&\hspace{-0.2cm}&\hspace{-0.2cm}\\
&\hspace{-0.2cm}1&\hspace{-0.2cm}&\hspace{-0.2cm}2&\hspace{-0.2cm}&\hspace{-0.2cm}3&\hspace{-0.2cm}\\
1&\hspace{-0.2cm}&\hspace{-0.2cm}2&\hspace{-0.2cm}&\hspace{-0.2cm}3&\hspace{-0.2cm}&\hspace{-0.2cm}4
\end{array}
\hspace{2cm}\begin{pmatrix}
&&+&\\
+&+&&\\
&+&+&+\\
+&+&+&+
\end{pmatrix}
\hspace{0.3cm}\rightarrow\hspace{0.3cm}
\begin{array}{ccccccc}
&\hspace{-0.2cm}&\hspace{-0.2cm}&\hspace{-0.2cm}3&\hspace{-0.2cm}&\hspace{-0.2cm}&\hspace{-0.2cm}\\
&\hspace{-0.2cm}&\hspace{-0.2cm}1&\hspace{-0.2cm}&\hspace{-0.2cm}2&\hspace{-0.2cm}&\hspace{-0.2cm}\\
&\hspace{-0.2cm}2&\hspace{-0.2cm}&\hspace{-0.2cm}3&\hspace{-0.2cm}&\hspace{-0.2cm}4&\hspace{-0.2cm}\\
1&\hspace{-0.2cm}&\hspace{-0.2cm}2&\hspace{-0.2cm}&\hspace{-0.2cm}3&\hspace{-0.2cm}&\hspace{-0.2cm}4
\end{array}\]

These examples do not correspond to ASMs, however, since they are not monotone triangles.

\begin{lemma}\label{row-increasing}Let $v = (v_1,\dots,v_n)$ be a vector of positive integers such that $v \preceq (n,n-1,\dots,2,1)$. Then there exists a row-increasing triangle $T$ of order $n$ such that $1,\dots,n$ occur as entries of $T$ $v_1,\dots,v_n$ times, respectively.\end{lemma}
\begin{proof}
The conjugate $x^*$ of a vector $x$ is the vector whose $k$th entry is equal to the number of entries of $x$ that are greater than or equal to $k$, for $k \in \{1,2,\dots,n\}$. Since $v \preceq (n,n-1,\dots,2,1)$, the Gale-Ryser theorem \cite{gale, ryser} implies that there exists a $(0,1)$-matrix $M$ with row-sum $(n,n-1,\dots,2,1)^*$ and column-sum $v$. Let $M'$ be the matrix resulting from reflecting $M$ about the horizontal axis.

Since $(n,n-1,\dots,2,1)^* = (n,n-1,\dots,2,1)$, this means that $M'$ has row-sum $(1,2,\dots,n)$ and therefore corresponds to a row-increasing triangle $T$. Each entry of $T$ equal to $k \in \{1,2,\dots,n\}$ corresponds to a $+$ entry in column $k$ of $M'$. Therefore, since $v$ is the column-sum of $M'$, the $k$th entry of $v$ is the number of occurrences of $k$ as an entry of $T$.
\end{proof}

We now prove that such a row-increasing triangle can be transformed into a monotone triangle with each entry occurring the same number of times by swapping positions of certain entries.

\begin{lemma}\label{monotone}Let $T$ be a row-increasing triangle of order $n$. Then there exists a monotone triangle of order $n$ with each entry occurring the same number of times as in $T$.\end{lemma}
\begin{proof}
We refer to a pair of entries $(b,a)$ in $T$ as an \emph{inversion} if $b>a$ and $a$ occurs directly north-east or south-east of $b$ (in the same north/south-east line). We call $(b,a)$ an \emph{upward} inversion if $a$ is north-east of $b$, and we call it a \emph{downward} inversion if $a$ is south-east of $b$. For example, inversions (upward and downward, respectively) are indicated in red in the following row-increasing triangles.

\[\begin{array}{ccccc}
&\hspace{-0.2cm}&\hspace{-0.2cm}{\color{red}1}&\hspace{-0.2cm}&\hspace{-0.2cm}\\
&\hspace{-0.2cm}{\color{red}2}&\hspace{-0.2cm}&\hspace{-0.2cm}3&\hspace{-0.2cm}\\
1&\hspace{-0.2cm}&\hspace{-0.2cm}2&\hspace{-0.2cm}&\hspace{-0.2cm}3
\end{array}
\hspace{2cm}
\begin{array}{ccccc}
&\hspace{-0.2cm}&\hspace{-0.2cm}{\color{red}3}&\hspace{-0.2cm}&\hspace{-0.2cm}\\
&\hspace{-0.2cm}1&\hspace{-0.2cm}&\hspace{-0.2cm}{\color{red}2}&\hspace{-0.2cm}\\
1&\hspace{-0.2cm}&\hspace{-0.2cm}2&\hspace{-0.2cm}&\hspace{-0.2cm}3
\end{array}\]

An entry $x$ in a triangle $T$ occurs \emph{before} another entry $y$ if there is a sequence of north-east and south-east steps from $x$ to $y$. Let $f(T)$ count the number of entries in $T$ occurring before an entry of less value. For example, in the following row-increasing triangles, all entries occurring before some entry of less value are indicated in red. The entries of less value are indicated in blue. Therefore $f(T)$, for each of these triangles respectively, is 8 and 6.

\[\begin{array}{ccccccccccc}
   &\hspace{-0.2cm}   &\hspace{-0.2cm}   &\hspace{-0.2cm}   &\hspace{-0.2cm}   &\hspace{-0.2cm}{\color{red}2}&\hspace{-0.2cm}   &\hspace{-0.2cm}   &\hspace{-0.2cm}   &\hspace{-0.2cm}   &\hspace{-0.2cm}   \\
   &\hspace{-0.2cm}   &\hspace{-0.2cm}   &\hspace{-0.2cm}   &\hspace{-0.2cm}{\color{red}2}&\hspace{-0.2cm}   &\hspace{-0.2cm}{\color{blue}1}&\hspace{-0.2cm}   &\hspace{-0.2cm}   &\hspace{-0.2cm}   &\hspace{-0.2cm}   \\
   &\hspace{-0.2cm}   &\hspace{-0.2cm}   &\hspace{-0.2cm}{\color{red}2}&\hspace{-0.2cm}   &\hspace{-0.2cm}{\color{red}3}&\hspace{-0.2cm}   &\hspace{-0.2cm}{5}&\hspace{-0.2cm}   &\hspace{-0.2cm}   &\hspace{-0.2cm}   \\
   &\hspace{-0.2cm}   &\hspace{-0.2cm}{\color{red}2}&\hspace{-0.2cm}   &\hspace{-0.2cm}{\color{red}3}&\hspace{-0.2cm}   &\hspace{-0.2cm}{4}&\hspace{-0.2cm}   &\hspace{-0.2cm}{5}&\hspace{-0.2cm}   &\hspace{-0.2cm}   \\
   &\hspace{-0.2cm}{1}&\hspace{-0.2cm}   &\hspace{-0.2cm}{\color{red}3}&\hspace{-0.2cm}   &\hspace{-0.2cm}{4}&\hspace{-0.2cm}   &\hspace{-0.2cm}{5}&\hspace{-0.2cm}   &\hspace{-0.2cm}{6}&\hspace{-0.2cm}   \\
{1}&\hspace{-0.2cm}   &\hspace{-0.2cm}{\color{red}2}&\hspace{-0.2cm}   &\hspace{-0.2cm}{3}&\hspace{-0.2cm}   &\hspace{-0.2cm}{4}&\hspace{-0.2cm}   &\hspace{-0.2cm}{5}&\hspace{-0.2cm}   &\hspace{-0.2cm}{6}\\
\end{array}
\hspace{2cm}
\begin{array}{ccccccccccc}
   &\hspace{-0.2cm}   &\hspace{-0.2cm}   &\hspace{-0.2cm}   &\hspace{-0.2cm}   &\hspace{-0.2cm}{\color{blue}1}&\hspace{-0.2cm}   &\hspace{-0.2cm}   &\hspace{-0.2cm}   &\hspace{-0.2cm}   &\hspace{-0.2cm}   \\
   &\hspace{-0.2cm}   &\hspace{-0.2cm}   &\hspace{-0.2cm}   &\hspace{-0.2cm}{\color{red}2}&\hspace{-0.2cm}   &\hspace{-0.2cm}{\color{blue}2}&\hspace{-0.2cm}   &\hspace{-0.2cm}   &\hspace{-0.2cm}   &\hspace{-0.2cm}   \\
   &\hspace{-0.2cm}   &\hspace{-0.2cm}   &\hspace{-0.2cm}{\color{red}2}&\hspace{-0.2cm}   &\hspace{-0.2cm}{\color{red}3}&\hspace{-0.2cm}   &\hspace{-0.2cm}{5}&\hspace{-0.2cm}   &\hspace{-0.2cm}   &\hspace{-0.2cm}   \\
   &\hspace{-0.2cm}   &\hspace{-0.2cm}{\color{red}2}&\hspace{-0.2cm}   &\hspace{-0.2cm}{\color{red}3}&\hspace{-0.2cm}   &\hspace{-0.2cm}{4}&\hspace{-0.2cm}   &\hspace{-0.2cm}{5}&\hspace{-0.2cm}   &\hspace{-0.2cm}   \\
   &\hspace{-0.2cm}{1}&\hspace{-0.2cm}   &\hspace{-0.2cm}{\color{red}3}&\hspace{-0.2cm}   &\hspace{-0.2cm}{4}&\hspace{-0.2cm}   &\hspace{-0.2cm}{5}&\hspace{-0.2cm}   &\hspace{-0.2cm}{6}&\hspace{-0.2cm}   \\
{1}&\hspace{-0.2cm}   &\hspace{-0.2cm}{2}&\hspace{-0.2cm}   &\hspace{-0.2cm}{3}&\hspace{-0.2cm}   &\hspace{-0.2cm}{4}&\hspace{-0.2cm}   &\hspace{-0.2cm}{5}&\hspace{-0.2cm}   &\hspace{-0.2cm}{6}\\
\end{array}\]

If $f(T) > 0$, then there exist entries $a_1 > a_k$ with $a_1$ occurring before $a_k$. The corresponding sequence of north/south-east steps $a_1, a_2, \dots, a_k$ has some $a_{i+1} > a_{i}$. If this was not the case, then $a_1 \leq a_2 \leq \dots \leq a_k$, and so $a_1 \leq a_k$, which contradicts that $a_1 > a_k$. Therefore if $f(T) > 0$, then $T$ contains an inversion. Note also that if $f(T) = 0$, then $T$ has no inversions and so $T$ is a monotone triangle.

We refer to each sub-triangle of order 2 in $T$ containing an inversion as an \emph{inverted triangle}. Since $T$ is row-increasing, each inverted triangle contains exactly one inversion. To see this, consider all possible configurations of an order 2 sub-triangle containing entries $1 \leq 2 \leq 3$.
\[\begin{array}{ccccccc}
 &\hspace{-0.2cm}2&\hspace{-0.2cm}&\\
1&\hspace{-0.2cm}&\hspace{-0.2cm}3
\end{array}
\hspace{0.5cm}
\begin{array}{ccccccc}
 &\hspace{-0.2cm}1&\hspace{-0.2cm}&\\
2&\hspace{-0.2cm}&\hspace{-0.2cm}3
\end{array}
\hspace{0.5cm}
\begin{array}{ccccccc}
 &\hspace{-0.2cm}3&\hspace{-0.2cm}&\\
1&\hspace{-0.2cm}&\hspace{-0.2cm}2
\end{array}
\hspace{0.5cm}
\begin{array}{ccccccc}
 &\hspace{-0.2cm}2&\hspace{-0.2cm}&\\
3&\hspace{-0.2cm}&\hspace{-0.2cm}1
\end{array}
\hspace{0.5cm}
\begin{array}{ccccccc}
 &\hspace{-0.2cm}1&\hspace{-0.2cm}&\\
3&\hspace{-0.2cm}&\hspace{-0.2cm}2
\end{array}
\hspace{0.5cm}
\begin{array}{ccccccc}
 &\hspace{-0.2cm}3&\hspace{-0.2cm}&\\
2&\hspace{-0.2cm}&\hspace{-0.2cm}1
\end{array}\]

The first three are the only sub-triangles that satisfy the row-increasing condition. The first has no inversions, and therefore is not an inverted triangle. The second and third have one inversion each; the upward inversion $(2,1)$ and the downward inversion $(3,2)$, respectively.

We call the number of rows in $T$ below an inverted triangle the \emph{height} of the inverted triangle. Some inverted triangles may have entries in common with others. We define the \emph{inverted trapezoids} of $T$ to be the union of overlapping inverted triangles of the same height. For example, the following row-increasing triangle has two inverted trapezoids, one of height 2 and one of height 1, indicated below in red.

\[\begin{array}{ccccccc}
&\hspace{-0.2cm}&\hspace{-0.2cm}&\hspace{-0.2cm}3&\hspace{-0.2cm}&\hspace{-0.2cm}&\hspace{-0.2cm}\\
&\hspace{-0.2cm}&\hspace{-0.2cm}1&\hspace{-0.2cm}&\hspace{-0.2cm}2&\hspace{-0.2cm}&\hspace{-0.2cm}\\
&\hspace{-0.2cm}2&\hspace{-0.2cm}&\hspace{-0.2cm}3&\hspace{-0.2cm}&\hspace{-0.2cm}4&\hspace{-0.2cm}\\
1&\hspace{-0.2cm}&\hspace{-0.2cm}2&\hspace{-0.2cm}&\hspace{-0.2cm}3&\hspace{-0.2cm}&\hspace{-0.2cm}4
\end{array}
:\hspace{2cm}
\begin{array}{ccccccc}
&\hspace{-0.2cm}&\hspace{-0.2cm}&\hspace{-0.2cm}{\color{red}3}&\hspace{-0.2cm}&\hspace{-0.2cm}&\hspace{-0.2cm}\\
&\hspace{-0.2cm}&\hspace{-0.2cm}{\color{red}1}&\hspace{-0.2cm}&\hspace{-0.2cm}{\color{red}2}&\hspace{-0.2cm}&\hspace{-0.2cm}\\
&\hspace{-0.2cm}2&\hspace{-0.2cm}&\hspace{-0.2cm}3&\hspace{-0.2cm}&\hspace{-0.2cm}4&\hspace{-0.2cm}\\
1&\hspace{-0.2cm}&\hspace{-0.2cm}2&\hspace{-0.2cm}&\hspace{-0.2cm}3&\hspace{-0.2cm}&\hspace{-0.2cm}4
\end{array}
\hspace{1cm}
\begin{array}{ccccccc}
&\hspace{-0.2cm}&\hspace{-0.2cm}&\hspace{-0.2cm}3&\hspace{-0.2cm}&\hspace{-0.2cm}&\hspace{-0.2cm}\\
&\hspace{-0.2cm}&\hspace{-0.2cm}{\color{red}1}&\hspace{-0.2cm}&\hspace{-0.2cm}{\color{red}2}&\hspace{-0.2cm}&\hspace{-0.2cm}\\
&\hspace{-0.2cm}{\color{red}2}&\hspace{-0.2cm}&\hspace{-0.2cm}{\color{red}3}&\hspace{-0.2cm}&\hspace{-0.2cm}{\color{red}4}&\hspace{-0.2cm}\\
1&\hspace{-0.2cm}&\hspace{-0.2cm}2&\hspace{-0.2cm}&\hspace{-0.2cm}3&\hspace{-0.2cm}&\hspace{-0.2cm}4
\end{array}\]

Within an inverted trapezoid, each entry is involved in at most one inversion. To see this, consider all possible ways for two inverted triangles of the same height to overlap, while still enforcing the row-increasing condition. Either a pair of upward inversions meet, an upward inversion is left of a downward inversion, or a pair of downward inversions meet.

\[\begin{array}{ccccc}
 &\hspace{-0.2cm}1&\hspace{-0.2cm}&\hspace{-0.2cm}2&\hspace{-0.2cm}\\
2&\hspace{-0.2cm}&\hspace{-0.2cm}3&\hspace{-0.2cm}&\hspace{-0.2cm}4
\end{array}
\hspace{1cm}
\begin{array}{ccccc}
 &\hspace{-0.2cm}1&\hspace{-0.2cm}&\hspace{-0.2cm}5&\hspace{-0.2cm}\\
2&\hspace{-0.2cm}&\hspace{-0.2cm}3&\hspace{-0.2cm}&\hspace{-0.2cm}4
\end{array}
\hspace{1cm}
\begin{array}{ccccc}
 &\hspace{-0.2cm}3&\hspace{-0.2cm}&\hspace{-0.2cm}4&\hspace{-0.2cm}\\
1&\hspace{-0.2cm}&\hspace{-0.2cm}2&\hspace{-0.2cm}&\hspace{-0.2cm}3
\end{array}\]

In all cases, all entries except one are involved in one inversion. Note that if a downward inversion $(3,2)$ were to be left of an upward inversion $(2,1)$, then the upper row of the inverted trapezoid would have $3$ occurring before $1$, which contradicts the row-increasing condition. The same logic as above can be extended to include inverted trapezoids consisting of more than two overlapping inverted triangles.

We \emph{switch} an inverted trapezoid by swapping the entries of each inversion in the trapezoid. Switching an inverted trapezoid results in another row-increasing triangle. To see this, we consider the same cases as above. Switching each inverted trapezoid above results in the following trapezoids, respectively.

\[\begin{array}{ccccc}
 &\hspace{-0.2cm}2&\hspace{-0.2cm}&\hspace{-0.2cm}3&\hspace{-0.2cm}\\
1&\hspace{-0.2cm}&\hspace{-0.2cm}2&\hspace{-0.2cm}&\hspace{-0.2cm}4
\end{array}
\hspace{1cm}
\begin{array}{ccccc}
 &\hspace{-0.2cm}2&\hspace{-0.2cm}&\hspace{-0.2cm}4&\hspace{-0.2cm}\\
1&\hspace{-0.2cm}&\hspace{-0.2cm}3&\hspace{-0.2cm}&\hspace{-0.2cm}5
\end{array}
\hspace{1cm}
\begin{array}{ccccc}
 &\hspace{-0.2cm}2&\hspace{-0.2cm}&\hspace{-0.2cm}3&\hspace{-0.2cm}\\
1&\hspace{-0.2cm}&\hspace{-0.2cm}3&\hspace{-0.2cm}&\hspace{-0.2cm}4
\end{array}\]

Note that if we were to swap the entries of inversions one at at time, we are not guaranteed to maintain the row-increasing condition. For example, swapping either inversion in the first inverted trapezoid above results in the entry 2 occurring twice in the same row. This is why it is necessary to always switch an inverted trapezoid in its entirety, since all resulting trapezoids above are row-increasing. Again, the same logic can be extended to include inverted trapezoids consisting of more than two overlapping inverted triangles.

Swapping any inversion $(y,x)$ reduces $f(T)$ in a row-increasing triangle $T$. To see this, consider $T$ and the triangle $T'$ resulting from swapping $(y,x)$. Any entry occuring before $x$ in $T'$ also occurred before $x$ in $T$, and any entry occurring before $y$ in $T'$ occurred before $x$ in $T$. Since $x < y$, this means that any entry contributing to $f(T')$ also contributed to $f(T)$, and swapping $x$ and $y$ means that $y$ no longer occurs before $x$ in $T'$. Since the inversion $(y,x)$ did contribute to $f(T)$, we therefore have that $f(T') < f(T)$.

Since switching an inverted trapezoid simply involves swapping multiple inversions, this means that switching an inverted trapezoid in $T$ always results in a row-increasing triangle $T'$ for which $f(T') < f(T)$. Therefore repeatedly switching inverted trapezoids in a row-increasing triangle eventually results in a row-increasing triangle $M$ for which $f(M) = 0$.

$M$ is a monotone triangle with each entry occurring the same number of times as in $T$.
\end{proof}

The proof of Conjecture \ref{proj_conj} follows from the previous lemmas.

\begin{theorem}\label{weighted projection}
Let $v$ be a vector of positive integers. Then there exists an $n\times n$ ASM whose weighted projection is $v$ if and only if $v \preceq (n,n-1,\dots,2,1)$..
\end{theorem}
\begin{proof}
Suppose $A$ is an $n \times n$ ASM whose weighted projection is $v$. Then, from Theorem 19 in \cite{brualdidahl}, it follows that $v \preceq (n,n-1,\dots,2,1)$.

Now suppose that $v \preceq (n,n-1,\dots,2,1)$. From Lemma \ref{row-increasing}, there exists a row-increasing triangle $T$ of order $n$ such that the $k$th entry of $v$ is the number of times $k$ occurs as an entry of $T$. Lemma \ref{monotone} then implies that we can use $T$ to construct a monotone triangle $T'$ with each entry occuring the same number of times as in $T$. We then construct the ASM $A$ corresponding to $T'$, which satisfies $v(A) = v$.
\end{proof}

Note that the proof of this theorem and the previous lemmas outline a construction of an ASM $A$ with given weighted projection $v(A)$. First, one constructs an $n \times n$ $(0,1)$-matrix with row-sum vector $(1,2,\dots,n)$ and column-sum vector $v(A)$ using the construction outlined in Krause's proof of the Gale-Ryser theorem \cite{krause}.

For example, to construct an ASM $A$ with $v(A) = (4,3,1,4,7,5,4)$, we first construct a $(0,1)$-matrix $S$ with row sum $(1,2,\dots,n)$ and column-sum $v(A)$. We do this by initialising $S$ to a $(0,1)$-matrix whose column sum is a permutation of $(1,2,\dots,n)$ with as many entries equal to $v(A)$ as possible. We then repeatedly swap two entries in the same row of $S$, a 1 in a column with a surplus and a 0 in a column with a deficit, until $S$ has the desired column-sum $v(A)$.

\hspace{0.5cm}$\phantom{v(}S\phantom{)}:\:\:\begin{pmatrix}
0&0&0&0&1&0&0\\
0&0&0&{\color{red}0}&1&0&{\color{red}1}\\
0&0&0&0&1&1&1\\
1&0&0&0&1&1&1\\
1&1&0&0&1&1&1\\
1&1&0&1&1&1&1\\
1&1&1&1&1&1&1
\end{pmatrix}
\rightarrow
\begin{pmatrix}
0&0&0&0&1&0&0\\
0&0&0&1&1&0&0\\
0&0&0&0&1&1&1\\
1&0&0&{\color{red}0}&1&1&{\color{red}1}\\
1&1&0&0&1&1&1\\
1&1&0&1&1&1&1\\
1&1&1&1&1&1&1
\end{pmatrix}
\rightarrow
\begin{pmatrix}
0&0&0&0&1&0&0\\
0&0&0&1&1&0&0\\
0&0&0&0&1&1&1\\
1&0&0&1&1&1&0\\
1&1&0&0&1&1&1\\
1&1&0&1&1&1&1\\
1&1&1&1&1&1&1
\end{pmatrix}$

\hspace{0.5cm}$v(S):\:\:\:\:\begin{pmatrix}
4&3&1&2&7&5&6
\end{pmatrix}
\hspace{0.175cm}\rightarrow\hspace{0.175cm}
\begin{pmatrix}
4&3&1&3&7&5&5
\end{pmatrix}
\hspace{0.175cm}\rightarrow\hspace{0.175cm}
\begin{pmatrix}
4&3&1&4&7&5&4
\end{pmatrix}\hspace{0.6cm}$

In general, this matrix implies a row-increasing triangle, which is then used to construct a monotone triangle using the construction of Lemma \ref{monotone}. This monotone triangle then corresponds to the desired ASM $A$.

Returning to our example, we next construct a row-increasing triangle from our $(0,1)$-matrix with row sum $(1,2,\dots,n)$ and column-sum $v(A)$. We then repeatedly switch inverted trapezoids until a monotone triangle is reached, and construct the ASM from our monotone triangle.

\[\begin{array}{ccccccccccccc}
 &\hspace{-0.2cm}     &\hspace{-0.2cm}   &\hspace{-0.2cm}   &\hspace{-0.2cm}   &\hspace{-0.2cm}   &\hspace{-0.2cm}{5}&\hspace{-0.2cm}   &\hspace{-0.2cm}   &\hspace{-0.2cm}   &\hspace{-0.2cm}   &\hspace{-0.2cm}  &\hspace{-0.2cm}  \\
   &\hspace{-0.2cm} &\hspace{-0.2cm}   &\hspace{-0.2cm}   &\hspace{-0.2cm}   &\hspace{-0.2cm}{4}&\hspace{-0.2cm}   &\hspace{-0.2cm}{5}&\hspace{-0.2cm}   &\hspace{-0.2cm}   &\hspace{-0.2cm}   &\hspace{-0.2cm}  &\hspace{-0.2cm}  \\
   &\hspace{-0.2cm} &\hspace{-0.2cm}   &\hspace{-0.2cm}   &\hspace{-0.2cm}{\color{red}5}&\hspace{-0.2cm}   &\hspace{-0.2cm}{\color{red}6}&\hspace{-0.2cm}   &\hspace{-0.2cm}{\color{red}7}&\hspace{-0.2cm}   &\hspace{-0.2cm}   &\hspace{-0.2cm}  &\hspace{-0.2cm}  \\
   &\hspace{-0.2cm} &\hspace{-0.2cm}   &\hspace{-0.2cm}{\color{red}1}&\hspace{-0.2cm}   &\hspace{-0.2cm}{\color{red}4}&\hspace{-0.2cm}   &\hspace{-0.2cm}{\color{red}5}&\hspace{-0.2cm}   &\hspace{-0.2cm}{\color{red}6}&\hspace{-0.2cm}   &\hspace{-0.2cm}  &\hspace{-0.2cm}  \\
  &\hspace{-0.2cm}  &\hspace{-0.2cm}{1}&\hspace{-0.2cm}   &\hspace{-0.2cm}{2}&\hspace{-0.2cm}   &\hspace{-0.2cm}{5}&\hspace{-0.2cm}   &\hspace{-0.2cm}{6}&\hspace{-0.2cm}   &\hspace{-0.2cm}{7}&\hspace{-0.2cm}  &\hspace{-0.2cm} \\
 &\hspace{-0.2cm}{1}&\hspace{-0.2cm}   &\hspace{-0.2cm}{2}&\hspace{-0.2cm}   &\hspace{-0.2cm}{4}&\hspace{-0.2cm}   &\hspace{-0.2cm}{5}&\hspace{-0.2cm}   &\hspace{-0.2cm}{6}&\hspace{-0.2cm}   &\hspace{-0.2cm}{7} &\hspace{-0.2cm}\\
1 &\hspace{-0.2cm}  &\hspace{-0.2cm}{2}&\hspace{-0.2cm}   &\hspace{-0.2cm}{3}&\hspace{-0.2cm}   &\hspace{-0.2cm}{4}&\hspace{-0.2cm}   &\hspace{-0.2cm}{5}&\hspace{-0.2cm}   &\hspace{-0.2cm}{6}&\hspace{-0.2cm}  &\hspace{-0.2cm} 7\\
\end{array}
\rightarrow
\begin{array}{ccccccccccccc}
 &\hspace{-0.2cm}     &\hspace{-0.2cm}   &\hspace{-0.2cm}   &\hspace{-0.2cm}   &\hspace{-0.2cm}   &\hspace{-0.2cm}{5}&\hspace{-0.2cm}   &\hspace{-0.2cm}   &\hspace{-0.2cm}   &\hspace{-0.2cm}   &\hspace{-0.2cm}  &\hspace{-0.2cm}  \\
   &\hspace{-0.2cm} &\hspace{-0.2cm}   &\hspace{-0.2cm}   &\hspace{-0.2cm}   &\hspace{-0.2cm}{4}&\hspace{-0.2cm}   &\hspace{-0.2cm}{5}&\hspace{-0.2cm}   &\hspace{-0.2cm}   &\hspace{-0.2cm}   &\hspace{-0.2cm}  &\hspace{-0.2cm}  \\
   &\hspace{-0.2cm} &\hspace{-0.2cm}   &\hspace{-0.2cm}   &\hspace{-0.2cm}{4}&\hspace{-0.2cm}   &\hspace{-0.2cm}{5}&\hspace{-0.2cm}   &\hspace{-0.2cm}{6}&\hspace{-0.2cm}   &\hspace{-0.2cm}   &\hspace{-0.2cm}  &\hspace{-0.2cm}  \\
   &\hspace{-0.2cm} &\hspace{-0.2cm}   &\hspace{-0.2cm}{1}&\hspace{-0.2cm}   &\hspace{-0.2cm}{5}&\hspace{-0.2cm}   &\hspace{-0.2cm}{6}&\hspace{-0.2cm}   &\hspace{-0.2cm}{7}&\hspace{-0.2cm}   &\hspace{-0.2cm}  &\hspace{-0.2cm}  \\
  &\hspace{-0.2cm}  &\hspace{-0.2cm}{1}&\hspace{-0.2cm}   &\hspace{-0.2cm}{2}&\hspace{-0.2cm}   &\hspace{-0.2cm}{5}&\hspace{-0.2cm}   &\hspace{-0.2cm}{6}&\hspace{-0.2cm}   &\hspace{-0.2cm}{7}&\hspace{-0.2cm}  &\hspace{-0.2cm} \\
 &\hspace{-0.2cm}{1}&\hspace{-0.2cm}   &\hspace{-0.2cm}{2}&\hspace{-0.2cm}   &\hspace{-0.2cm}{4}&\hspace{-0.2cm}   &\hspace{-0.2cm}{5}&\hspace{-0.2cm}   &\hspace{-0.2cm}{6}&\hspace{-0.2cm}   &\hspace{-0.2cm}{7} &\hspace{-0.2cm}\\
1 &\hspace{-0.2cm}  &\hspace{-0.2cm}{2}&\hspace{-0.2cm}   &\hspace{-0.2cm}{3}&\hspace{-0.2cm}   &\hspace{-0.2cm}{4}&\hspace{-0.2cm}   &\hspace{-0.2cm}{5}&\hspace{-0.2cm}   &\hspace{-0.2cm}{6}&\hspace{-0.2cm}  &\hspace{-0.2cm} 7\\
\end{array}
\implies
\begin{pmatrix}
&&&&+&&\\
&&&+&&&\\
&&&&&+&\\
+&&&-&&&+\\
&+&&&&&\\
&&&+&&&\\
&&+&&&&
\end{pmatrix}\]

A direct consequence of Theorem \ref{weighted projection} is the following.
\begin{corollary}
There exists a monotone triangle for which $1,2,\dots,n$ occur as entries $v_1,v_2,\dots,v_n$ times if and only if $(v_1,v_2,\dots,v_n) \preceq (n,n-1,\dots,2,1)$.
\end{corollary}

Note that similar results have previously been proven in relation to a generalisation of the Birkhoff polytope of doubly stochastic matrices called the \emph{alternating sign matrix polytope} \cite{behrend}.

\section{The alternating sign matrix polytope}

The alternating sign matrix polytope $ASM_n$ is the set of $n \times n$ matrices with real entries for which all row and column sums are 1, and all partial row and column sums (extending from either end of a row or column) are non-negative. Equivalently, $ASM_n$ is the convex hull of the $n \times n$ alternating sign matrices in the space $\mathbb{R}^{n \times n}$ of all real $n \times n$ matrices.

Striker \cite{striker} proved that the weighted projection of all elements of the ASM polytope form the \emph{permutohedron} (Theorem 3.8). In this context, the \emph{integer points} of the permutohedron are exactly those positive integer vectors which are majorized by $z_n = (n, \dots, 2, 1)$.

Our result, Theorem \ref{weighted projection}, would follow directly from Striker's result if ASMs were the only elements of the ASM polytope whose weighted projections were integer points of the permutohedron. This is not the case, however. Therefore Striker's result does not guarantee that all integer points in the permutohedron are the weighted projection of some ASM, but rather that all are the weighted projection of some matrix in the ASM polytope.

\begin{example}\label{polytope elements} The following are some examples of matrices in $ASM_3$ with weighted projection $(2, 2, 2)$.
\[\begin{pmatrix}
\frac{1}{3}&\frac{1}{3}&\frac{1}{3}\\[6pt]
\frac{1}{3}&\frac{1}{3}&\frac{1}{3}\\[6pt]
\frac{1}{3}&\frac{1}{3}&\frac{1}{3}
\end{pmatrix}
\hspace{1cm}
\begin{pmatrix}
\frac{1}{3}&\frac{1}{4}&\frac{5}{12}\\[6pt]
\frac{1}{3}&\frac{1}{2}&\frac{1}{6}\\[6pt]
\frac{1}{3}&\frac{1}{4}&\frac{5}{12}
\end{pmatrix}
\hspace{1cm}
\begin{pmatrix}
\frac{1}{6}&\frac{2}{3}&\frac{1}{6}\\[6pt]
\frac{2}{3}&-\frac{1}{3}&\frac{2}{3}\\[6pt]
\frac{1}{6}&\frac{2}{3}&\frac{1}{6}
\end{pmatrix}
\hspace{1cm}
\begin{pmatrix}
\frac{1}{16}&\frac{3}{4}&\frac{3}{16}\\[6pt]
\frac{7}{8}&-\frac{1}{2}&\frac{5}{8}\\[6pt]
\frac{1}{16}&\frac{3}{4}&\frac{3}{16}
\end{pmatrix}
\]
\end{example}

Armstrong and McKeown \cite{mckeown} define a surjection from $n \times n$ ASMs to the integer points of the permutohedron of order $n-1$ (Proposition 3), which takes an ASM $A$ and sends it to the \emph{center of mass} of its \emph{generalised Waldspurger matrix} $WT(A)$. The center of mass of $WT(A)$ is simply its column-sum vector.

For example, the following ASMs both have weighted projection $(3,2,3,2)$ and the center of mass of their generalised Waldspurger matrices are both $(1,2,1)$.
\[A = \begin{pmatrix}
&+&&\\
+&-&+&\\
&&&+\\
&+&&
\end{pmatrix}
\hspace{1cm}
WT(A) = \begin{pmatrix}
1&&\\
&1&\\
&1&1
\end{pmatrix}\]
\[B = \begin{pmatrix}
&&+&\\
+&&&\\
&+&-&+\\
&&+&
\end{pmatrix}
\hspace{1cm}
WT(B) = \begin{pmatrix}
1&1&\\
&1&\\
&&1
\end{pmatrix}\]

While this result may be more analogous to our own, the methods of proof are distinct. Armstrong and McKeown focus on Waldspurger matrices and employ a number of geometric techniques. Additionally, as previously outlined, the proof of Theorem \ref{weighted projection} gives an explicit construction for an ASM with given weighted projection.

Example \ref{polytope elements} and the prior discussion of the difference between Striker's result and our own motivates the study of the pre-images of integer points of the permutohedron of order $n$. In order to understand how elements of the ASM polytope with the same weighted projection relate to one another, we define a \emph{T-block} \cite{finite order} as follows. An $n \times n$ T-block $T_{i_1,j_1:\,i_2,j_2}$ is a matrix $\pm[t_{ij}]$ such that
\[t_{ij} =  \begin{cases} 
      \color{white}-\color{black}1  & \text{ if } (i,j) \in \{(i_1,j_1), (i_2,j_2)\} \\
      -1 & \text{ if } (i,j) \in \{(i_2,j_1), (i_1,j_2)\} \\
      \color{white}-\color{black}0  & \text{otherwise} 
   \end{cases}
\] where $i_1 < i_2$, and $j_1 < j_2$. A T-block is most usefully visualised as an $n \times n$ matrix whose non-zero entries form a (not necessarily contiguous) copy of the following  \cite{patterns}.
\[\pm\begin{pmatrix}
\hspace{0.3cm}1&-1\\
-1&\hspace{0.3cm}1
\end{pmatrix}\]

All $n \times n$ ASMs can be obtained from one another through the addition of T-blocks \cite{bruhat}. The following is the analogous result for elements of the ASM polytope of order $n$.

\begin{lemma}\label{T-block}Let $A,B \in ASM_n$. Then $B$ can be expressed as the sum of $A$ and scalar multiples of T-blocks.\end{lemma}
\begin{proof}
If $A = B$, this is trivial. We therefore assume $A \not= B$ and consider $D = B-A$. Let $i$ be the lowest index of a row of $D$ containing a non-zero entry, and let $j$ be the lowest index of a non-zero entry in row $i$. Since $A$ and $B$ have all row sums equal to 1, $D$ has all row sums equal to zero, and there must therefore be some entry $D_{ij'}$ in row $i$ of opposite sign to $D_{ij}$. Similarly, there must be some entry $D_{i'j}$ in column $j$ of opposite sign to $D_{ij}$.

Let $c_1 = D_{ij}$, $T^{(1)} = T_{i,j;\; i'j'}$, and $D' = D - c_1T^{(1)}$. In every row/column that we have added $c_1$ to an entry, we have also subtracted $c_1$ from another entry. Therefore $D'$ is a matrix with all row and column sums equal to zero, with the first non-zero entry occurring in a later position than in $D$. We repeat this step, taking note of $c_k$ and $T^{(k)}$ on step $k$ and redefining $D'$ to be a matrix with row/column sums equal to zero and with its first non-zero entry moving to a later position in each step. We therefore eventually reach a zero matrix $D'$ on some step $m$, and we can rearrange to find
\[B = A + c_1T^{(1)} + \dots + c_mT^{(m)},\]
for T-blocks $T^{(1)}, \dots, T^{(m)}$.
\end{proof}

\begin{example}\label{asm plus T-blocks} Consider the alternating sign matrix
\[A = \begin{pmatrix}
&+&\\
+&-&+\\
&+&
\end{pmatrix},\]
 and the following matrices from Example \ref{polytope elements}.
\[\begin{pmatrix}
\frac{1}{3}&\frac{1}{3}&\frac{1}{3}\\[6pt]
\frac{1}{3}&\frac{1}{3}&\frac{1}{3}\\[6pt]
\frac{1}{3}&\frac{1}{3}&\frac{1}{3}
\end{pmatrix}
= A + \frac{1}{3}(T_{1,1;\;2,2} - T_{1,2;\;2,3} + T_{2,2;\;3,3} - T_{2,1;\;3,2})\]
\[\begin{pmatrix}
\frac{1}{16}&\frac{3}{4}&\frac{3}{16}\\[6pt]
\frac{7}{8}&-\frac{1}{2}&\frac{5}{8}\\[6pt]
\frac{1}{16}&\frac{3}{4}&\frac{3}{16}
\end{pmatrix}
= A + \frac{1}{16}(T_{1,1;\;2,2} - T_{2,1;\;3,2}) + \frac{3}{16}(T_{2,2;\;3,3} - T_{1,2;\;2,3})\]
\end{example}

We define the \emph{depth} $d(T)$ of a T-block as follows.
\[d(T) =  \begin{cases} 
      i_2-i_1 & \text{ if } T = +T_{i_1,j_1:\,i_2,j_2}\\
      i_1-i_2 & \text{ if } T = -T_{i_1,j_1:\,i_2,j_2}
   \end{cases}
\] 

We say that T-blocks $T$ and $T'$ have \emph{opposite} depth if $d(T) = -d(T')$.

We now characterise all elements of the ASM polytope that are mapped by weighted projection to the same point of the permutohedron.

\begin{theorem}\label{same projection}Let $A,B \in ASM_n$. Then $v(A) = v(B)$ if and only if $B$ can be expressed as the sum of $A$ and scalar multiples of pairs of T-blocks of opposite depth.\end{theorem}
\begin{proof}
Suppose $B$ can be expressed as such a sum. Then, taking the weighted projection of both sides:
\[v(B) = v\big(A + c_1(T_1 + T'_1) + \dots + c_k(T_k + T'_k)\big),\]
where $d(T_i) = -d(T'_i)$. Since the weighted projection is a linear transformation,
\[v(B)  = v(A) + c_1\big(v(T_1) + v(T'_1)\big) + \dots + c_k\big(v(T_k) + v(T'_k)\big).\]
Since $T_i$ and $T'_i$ have opposite depth, each $\big(v(T_i) + v(T'_i)\big)$ is equal to zero. Therefore 
\[v(B) = v(A) + c_1\cdot0 + \dots + c_k\cdot0 = v(A).\]

Now suppose $v(A) = v(B)$. If $A=B$, this is trivially true. Therefore assume $A \not= B$ and let $D=B-A$. Note that all row and column sums of $D$ are zero, and since $v(A) = v(B)$,
\[v(D) = v(B-A) = v(B)-v(A) = (0,0,\dots,0).\]

Since $A \not=B$, this means that $D$ contains non-zero entries. Let $j$ be the index of the first column of $D$ with non-zero entries, and let $i$ be the index of the first entry in column $j$ that is non-zero. Since all line sums are zero, there is at least one entry $D_{i'j}$ in column $j$ and another entry $D{ij'}$ in row $i$ which have opposite sign to $D_{ij}$.

Without loss of generality, assume $D_{i,j}$ is positive. If all negative entries in column $i$ occurred after all positive entries, then the entry of $v(D)$ corresponding to column $i$ would be negative. Since $v(D) = 0$, this is false. Therefore there are at least two positive entries $D_{ij}$ and $D_{pj}$ in column $j$ of $D$, with a negative entry $D_{i'j}$ between them. Let $p'$ be such that $p - p' = -(i - i')$. Note that $p' > i$, since $i'$ occurs between $i$ and $p$.

Let $T^{(1)} = T_{i,j;i'j'}$ and $S^{(1)} = T_{p,j;p',j'}$, and note that $T^{(1)}$ has opposite depth to $S^{(1)}$, since $p - p' = -(i - i')$. Let $c_1 = D_{ij}$, and let $D' = D - c_1(T^{(1)}+S^{(1)})$. In every row or column that we have added $c_1$ to an entry, we have also subtracted it from another entry. Therefore $D'$ is a matrix with all line sums equal to zero, and with the first non-zero entry occurring in a later position than in $D$. We repeat this step, taking note of $c_k$, $T^{(k)}$, and $S^{(k)}$ on step $k$, and redefining $D'$ to be a matrix with line sums equal to zero and with its first non-zero entry moving to a later position in each step. We therefore eventually reach a zero matrix $D'$ on some step $m$, and we can rearrange to find
\[B = A + c_1(T^{(1)}+S^{(1)}) + \dots + c_m(T^{(m)}+S^{(m)}),\]
where each $T^{(k)}$ and $S^{(k)}$ are a pair of T-blocks of opposite depth.
\end{proof}

For example, all matrices in Example \ref{asm plus T-blocks} have weighted projection $(2,2,2)$. It was also shown that each matrix in the example can be obtained from the ASM $A$ by the addition of scalar multiples of pairs of T-blocks of opposite depth.

A direct consequence of Theorem \ref{same projection} is the following result, concerning the integer points of the permutohedron.
\begin{corollary}
Let $A$ be an $n \times n$ alternating sign matrix and $M \in ASM_n$. Then $v(M) = v(A)$ if and only if $M$ can be expressed as the sum of $A$ and scalar multiples of pairs of T-blocks of opposite depth.
\end{corollary}

Note that elements of $ASM_n$ with the same weighted projection have a similar characterisation to that of alternating sign hypermatrices corresponding to the same Latin-like square \cite{ASHM-decomp}. We now consider the limitations of Theorem \ref{weighted projection} for the characterisation of ASHLs.


\section{Returning to the characterisation of ASHLs}
An immediate consequence of Theorem \ref{weighted projection} is the following.
\begin{corollary}\label{major_cor}
Let $v$ be a row or column of an ASHL. Then $v \preceq (n,n-1,\dots,2,1)$.
\end{corollary}

This means the integer that can occur most as an entry in an ASHL is $\frac{n+1}{2}$ if $n$ odd, and $\frac{n}{2}$ or $\frac{n+2}{2}$ if $n$ even. However, Corollary \ref{major_cor} is not a sufficient condition for a given $n \times n$ array of positive integers to be an ASHL. For example, the following array has each row and column majorized by $(n,n-1,\dots,2,1)$, but is not an ASHL.

\[\begin{array}{|c|c|c|}
\hline
    2 & 2 & 2 \\
\hline
    3 & 2 & 1 \\
\hline
    1 & 2 & 3\\
\hline
\end{array}\]

This is because the outer rows and columns of an ASHL must contain each entry exactly once. This follows from the fact that the outer rows and columns of an ASM contain exactly one non-zero entry, which is $+1$.

More generally, the number of non-zero entries in row/column $1,2,3,\dots, n-2, n-1,n$ of an ASM is bounded above by $1,3,5,\dots,5,3,1$. Extrapolating from this, one might posit that the number of times a particular number can occur as an entry in row/column $1,2,3,\dots, n-2, n-1,n$ of an ASHL is also bounded above by $1,3,5,\dots,5,3,1$. 

This is not the case, however. The construction given in the proof of Theorem 3.3 in \cite{ASHM-decomp} defines an $n \times n$ ASHL for all $n \in \mathbb{N}$ which exceeds this bound by 2 in rows/columns $2, 3, \dots, \lceil\frac{n-4}{2}\rceil$ and $\lceil\frac{n+5}{2}\rceil, \dots, n-2, n-1$. An example of this construction is the following $7 \times 7$ ASHL, with the same entry occurring 5 times in row/column 2 and 6.

\[\begin{array}{|c|c|c|c|c|c|c|}
\hline
6&3&1&4&7&5&2\\
\hline
3&\SmallArray{1-3\\+6}&4&\SmallArray{3-4\\+5}&4&\SmallArray{2-5\\+7}&5\\
\hline
1&4&\SmallArray{3-4\\+5}&\SmallArray{2-3+4\\-5+6}&\SmallArray{3-4\\+5}&4&7\\
\hline
4&\SmallArray{3-4\\+5}&\SmallArray{2-3+4\\-5+6}&\SmallArray{+1-2\\+3-4+5\\-6+7}&\SmallArray{2-3+4\\-5+6}&\SmallArray{3-4\\+5}&4\\
\hline
7&4&\SmallArray{3-4\\+5}&\SmallArray{2-3+4\\-5+6}&\SmallArray{3-4\\+5}&4&1\\
\hline
5&\SmallArray{2-5\\+7}&4&\SmallArray{3-4\\+5}&4&\SmallArray{1-3\\+6}&3\\
\hline
2&5&7&4&1&3&6\\
\hline
\end{array}
\hspace{1cm}\rightarrow\hspace{1cm}
\begin{array}{|c|c|c|c|c|c|c|}
\hline
6&3&1&4&7&5&2\\
\hline
3&4&4&4&4&4&5\\
\hline
1&4&4&4&4&4&7\\
\hline
4&4&4&4&4&4&4\\
\hline
7&4&4&4&4&4&1\\
\hline
5&4&4&4&4&4&3\\
\hline
2&5&7&4&1&3&6\\
\hline
\end{array}\]

For simplicity, assume $n$ is odd, let  $k = \frac{n+1}{2}$, and consider some $n \times n \times n$ ASHM $A$ with corresponding ASHL $L(A)$. Note that the following analysis of columns can similarly be applied to rows or to ASHMs of even order.

From the previous example, we know it is possible to have $k$ occurring as an entry $n$ times in the central column. The following example demonstrates that it is also possible to have $k$ occurring $n$ times in an off-centre column, since 4 occurs as every entry in the third column.

\[\begin{array}{|c|c|c|c|c|c|c|}
\hline
 3 & 5 & 4 & 7 & 1 & 2 & 6\\
\hline
 1 & 4 & \SmallArray{3-4\\+5} & 4 & 6 & 7 & 2\\
\hline
 7 & 3 & \SmallArray{1-3\\+6} & 3 & 2 & 4 & 5\\
\hline
 5 & 7 & 4 & 2 & 3 & 6 & 1\\
\hline
 2 & 6 & \SmallArray{3-6\\+7} & 6 & 5 & 1 & 4\\
\hline
 4 & 1 & \SmallArray{2-4\\+6} & 5 & 4 & 3 & 7\\
\hline
 6 & 2 & 4 & 1 & 7 & 5 & 3\\
\hline
\end{array}
\hspace{1cm}\rightarrow\hspace{1cm}
\begin{array}{|c|c|c|c|c|c|c|}
\hline
 3 & 5 & 4 & 7 & 1 & 2 & 6\\
\hline
 1 & 4 & 4 & 4 & 6 & 7 & 2\\
\hline
 7 & 3 & 4 & 3 & 2 & 4 & 5\\
\hline
 5 & 7 & 4 & 2 & 3 & 6 & 1\\
\hline
 2 & 6 & 4 & 6 & 5 & 1 & 4\\
\hline
 4 & 1 & 4 & 5 & 4 & 3 & 7\\
\hline
 6 & 2 & 4 & 1 & 7 & 5 & 3\\
\hline
\end{array}
\]

This is not possible for all off-centre columns, however. Trivially, we have seen it is not possible for columns 1 or $n$, but it is also not possible for $k$ to to occur $n$ times in column 2 or $n-1$ of $L(A)$.

To see why this is true, consider a column of $L(A)$ with $k$ occurring $n$ times. Since the first and last terms of the column in $A$ are $k$, then the second and second last terms of $A$ cannot contain any negative entry besides $-k$. Therefore this column contains at least 5 occurrences of $\pm k$. Since the second and second last columns in any ASM have at most 3 non-zero entries, these columns of $L(A)$ cannot contain the same entry $n$ times.

Although $k$ cannot occur $n$ times in column 2 or $n-1$ of $L(A)$, it is possible for $k$ to occur $n-2$ times, as the following example demonstrates.

\[\begin{array}{|c|c|c|c|c|c|c|}
\hline
 1 & 3 & 2 & 4 & 7 & 5 & 6\\
\hline
 2 & 4 & 1 & 3 & 5 & 6 & 7\\
\hline
 3 & \SmallArray{1-3\\+6}  & 3 & 2 & 4 & 7 & 5\\
\hline
 4 & \SmallArray{3-4\\+5}  & 4 & 7 & 6 & 1 & 2\\
\hline
 5 & \SmallArray{2-5\\+7}  & 5 & 6 & 3 & 4 & 1\\
\hline
 6 & 4 & 7 & 5 & 1 & 2 & 3\\
\hline
 7 & 5 & 6 & 1 & 2 & 3 & 4\\
\hline
\end{array}
\hspace{1cm}\rightarrow\hspace{1cm}
\begin{array}{|c|c|c|c|c|c|c|}
\hline
1 & 3 & 2 & 4 & 7 & 5 & 6\\
\hline
 2 & 4 & 1 & 3 & 5 & 6 & 7\\
\hline
 3 & 4 & 3 & 2 & 4 & 7 & 5\\
\hline
 4 & 4 & 4 & 7 & 6 & 1 & 2\\
\hline
 5 & 4 & 5 & 6 & 3 & 4 & 1\\
\hline
 6 & 4 & 7 & 5 & 1 & 2 & 3\\
\hline
 7 & 5 & 6 & 1 & 2 & 3 & 4\\
\hline
\end{array}
\]

Returning to the general case of $n$ odd or even, let $k = \frac{n+1}{2}$ and let $m_i$ be the maximum possible number of occurrences of $\lfloor k \rfloor$ in row $i$ or column $i$ of any $n \times n$  ASHL. Using the previous analysis, we have the following.
\[m_1 = 1,\:\: n-2 \leq m_2 \leq n-1,\:\: n-2 \leq m_3 \leq n,\:\: \dots,\:\: n-2 \leq m_{\lfloor k \rfloor-1} \leq n,\:\: m_{\lfloor k \rfloor} = n,\]
\[m_{\lceil k \rceil} = n,\:\: n-2 \leq m_{\lceil k \rceil+1} \leq n,\:\: \dots,\:\: n-2 \leq m_{n-2} \leq n,\:\: n-2 \leq m_{n-1} \leq n-1,\:\: m_n = 1\]

Any of our attempts to maximise the number of occurences of $k$ in one particular off-centre row or column has resulted in a reduction of the total number of occurrences of $k$ throughout the ASHL. While the bounds described above apply to an individual row or column, there appears to be a more global restriction on the total number of occurrences of $k$ that has yet to be characterised.

As previously discussed in \cite{ASHM-decomp}, the theoretical best construction would result in an $n \times n$ ASHL wherein $k$ occurs as an entry $(n-2)^2 + 4$ times. This would be achieved by each entry in the inner $(n-2) \times (n-2)$ square being equal to $k$, as well as one entry in each of the outer rows/columns. As in that paper, we are still looking for a construction that achieves this for $n > 7$ or an argument that proves that this is not possible.


\end{document}